\documentclass[dvipdfmx]{amsart}

\usepackage{amsmath, amsfonts,amssymb}
\usepackage{graphicx}
\usepackage{url}
\usepackage[dvipdfm, usenames]{color}
\usepackage{mathrsfs}

\usepackage{comment}
\usepackage[dvipdfmx, 
colorlinks=true       % /true リンクに色をつけるか、否か（デフォルトはfalse）
]{hyperref}

\usepackage{color}
\usepackage{soul}

\usepackage{amscd}
\newtheorem{Def}{Definition}[section]

\newtheorem{Lem}[Def]{Lemma}
\newtheorem{Thm}[Def]{Theorem}

\newtheorem{Rem}[Def]{Remark}
\newtheorem{Ass}[Def]{Assumption}

\newcommand{\dimY}{m}
\newcommand{\dimB}{k}
\newcommand{\dimX}{d}
\newcommand{\dimZ}{ {\dimY \times \dimB} }

\newcommand{\ds}{\, \mathrm{d}s }
\newcommand{\dBs}{\, \mathrm{d}W(s)}

\newcommand{\NewtonF}{\mathrm{F}}
\author{Dai Taguchi}
\address[D.~Taguchi]{Osaka University, 1-3, Machikaneyama-cho, Toyonaka, Osaka 560-8531, Japan}
\email{dai.taguchi.dai@gmail.com}

\author{Takahiro Tsuchiya$^{*}$}
\thanks{$*$ Corresponding author}
\address[T.~Tsuchiya]{School of Computer Science and Engineering, The University of Aizu}
\email[Corresponding author]{suci@probab.com}

\title{
Newton-Kantorovitch method for decoupled forward-backward stochastic differential equations
} 
\begin{document}
\begin{abstract}
We present and prove a Newton-Kantorovitch method for solving decoupled forward-backward stochastic differential equations 
(FBSDEs) involving smooth coefficients with uniformly bounded derivatives. 
As Newton's method is required a suitable initial condition to converge, we show that such initial conditions are solutions of a linear backward stochastic differential equation. 
In addition, we show that converges linearly to the solution.

\end{abstract}
%\MSC[2010] 
%\end{keyword}

%\end{frontmatter}

%\linenumbers

\maketitle

\section{Introduction}
In this paper, we study a method of approximating non-Markovian and decoupled forward-backward stochastic differential equations (FBSDEs) of the following form for arbitrary $T > 0$:
\begin{equation}\label{fbsde1 decoupled}
\begin{split}
	X (t) &= X(0)+ \int_0^t  b (s,X(s))  \ds +\int_0^t \sigma (s, X (s))  \dBs, \\
	Y (t) &=  \varphi \left(X(T) \right) +\int_t^T f (s,X(s),Y(s),Z(s))  \ds - \int_t^T Z(s)\dBs, 
\end{split}
\end{equation}
where the solution triples $(X,Y,Z)\equiv  \left(X(t), Y(t), Z(t) \right)_{t \in [0, T]}$ 
take values in $\mathbb{R}^\dimX \times \mathbb{R}^{\dimY} \times \mathbb{R}^{\dimZ}$, 
$W$ is a $\dimB$-dimensional Wiener process, and $b, f$, $\sigma$, and $\varphi$ are measurable functions
that could in general be random and defined on a probability space. 

We propose a scheme for approximating such FBSDEs with random coefficients based on applying the Newton-Kantorovitch theory. %\cite{Kantorovitch1939, kantorovich1948newton} 
Unlike a four-step scheme that relies on the Markov structure 
(see e.g., Ma, Protter and Yong \cite{Ma1994fourstepscheme}, and Delarue \cite{Delarue2002209}), 
this method allows us to find a non-Markovian approximation. 
In addition, as our approach is a type of contraction mapping \cite{antonelli1993, PardouxTang1999}, 
the approximation works for arbitrary large durations with smooth random coefficients, and with monotonicity conditions; 
for details,  see Hu and Peng \cite{HuPeng1995}, Peng and Wu \cite{PengWu1999Themethodofcontinuation}, and Yong \cite{Yong1997methodofcontinuation}. 
We also refer readers to \cite{Zhang2017zbMATH06738088, MaZhang2012} 
for more details about the uniqueness and the existence of FBSDE solutions. 

Most numerical algorithms for FBSDEs are based on time-space discretization schemes \cite{delarue2006}
for quasi-linear parabolic partial differential equations 
for coupled FBSDEs with monotonicity condition \cite{benderZhang2008}.
In contrast, to the best of our knowledge, very little has been published on state space discretization based approaches. 
Vidossich proved that the Chaplygin and Newton methods are equivalent for ordinary differential equations \cite{Vidossich1978188}, 
and Kawabata and Yamada studied a state space discretization based on Newton-Kantorovitch approach 
for stochastic differential equations (SDEs) \cite{KYnewton91zbMATH00019571}. 
Ouknine \cite{Ouknine1993doi:10.1080/17442509308833863} showed that the coefficients can be relaxed by a linear growth condition. 
Amano %provided in \cite{Amano09zbMATH05574278} the rate of convergence for Newton-Kantorovitch method and 
obtained a probabilistic second-order error estimate \cite{Amano12zbMATH06028640}.
Wrzosek \cite{Monika2012MR2967195} extended to stochastic functional differential equations. 

The aim of this paper is 
to formulate Newton-Kantorovitch  approximation of the decoupled multi-dimensional FBSDEs with random coefficients  and 
to show that it converges globally \eqref{Theorem: the linear convergence inequality FBSDE}
 as follows: % in setting.  
\begin{Thm}\label{Theorem: the convergence of the Newton-Kantorovitch approximation process3}
Suppose that 
$b, \sigma, f$, and $\varphi$ are all $C^1$, their derivatives are uniformly bounded $(s,\omega)$-a.e.,
and $b, \sigma$, and  $f$ are square-integrable w.r.t. their time variables. 
Then, there exists a positive constant $C>0$ such that 
\begin{equation}\label{Theorem: the linear convergence inequality FBSDE}
\begin{split}
	\| (X -{X}_{n+1},Y - Y_{n+1}, Z - Z_{n+1} )\|  
	\leq  C2^{-n}
	\quad n \in \mathbb{N} \cup \{0\}. %^{\epsilon^{-1}C},
\end{split}
\end{equation}
\end{Thm}
As pointed out in \cite{KarouiPengel1997backward, BahlaliEddahbiandOuknine2017AOP}, 
solutions $(X,Y,Z)$ of decoupled FBSDEs can be viewed as fixed points of a mapping. 
Here, we consider an alternative mapping $\NewtonF_{\varphi} : \mathbb{S}^{2}_{\dimX} \times \mathbb{S}^{2}_{\dimY} \times \mathbb{H}^2 \to \mathbb{S}^{2}_{\dimX} \times \mathbb{L}^2_T$ inspired by a one-dimensional analog of \cite{KYnewton91zbMATH00019571, Niwa2003}, as follows. 
For given $\varphi$ and $u=(x, y, z)\in \mathbb{S}^{2}_{\dimX} \times \mathbb{S}^{2}_{\dimY} \times \mathbb{H}^2$, %with an initial condition, 
we define 
\begin{equation}
\begin{split}\label{Def: F in BSDEs}
	\NewtonF_{\varphi} (u)(t)
	= 
	\begin{pmatrix}
	\displaystyle x(t) - x(0) -\int_0^t b(s, x(s)) \ds - \int_0^t  \sigma(s,x(s)) \dBs\\
	\displaystyle y (t) - \varphi(x(T))  -\int_t^T f (s, u(s)) \ds + \int_t^T  z (s) \dBs 
	\end{pmatrix}^\top
	.%\in \mathbb{S}^2_{\dimX} \times \mathbb{L}_T^2.
\end{split}
\end{equation}
%Furthermore,
Then, for any $(X_{0}, Y_{0}, Z_{0}) \in \mathbb{S}^{2}_{\dimX} \times \mathbb{S}^{2}_{\dimY} \times \mathbb{H}^2$ with $X_0(0)=X(0)$ and assuming that the driver $f$ is smooth, the Newton-Kantorovitch approximation process is given by 
\begin{equation}
\begin{split}\label{equation: backward Newton-Kantorovitch def}
	\NewtonF_{\varphi}(X_{n},Y_{n},Z_{n}) +\NewtonF_{\varphi}' (X_{n},Y_{n},Z_{n}) (X_{n+1}-X_{n},Y_{n+1}-Y_{n},Z_{n+1}-Z_{n}) =0, 
	\\ X_{n+1}(0)=X (0), 
\end{split}
\end{equation}
%with initial condition $X_{n+1}(0)=X_{n}(0)$ 
for $n \in \mathbb{N} \cup \{0\}$,  
where $\NewtonF'_{\varphi}$ stands for G\^ateaux derivative. 
This sequence is well-posed iff a unique system of linear backward stochastic differential equations (BSDEs) formally by Theorem \ref{Theorem: linear BSES characerlization}. 
Note that, if the given duration $T$ is arbitrarily, then, even for linear FBSDEs with constant coefficients, 
the necessary and sufficient conditions become more complicated 
when the diffusion coefficients also depend on $z$, 
as was pointed out by Ma, Wu, Zhang and Zhang \cite{FBSDEMaWuZhangZhangzbMATH06464848}. 
%We extend the result to the general case, that is, the FBSDE is truly coupled, see also the result of Ma, Yin and Zhang \cite{MaZhang2012}. 
%However, we are interested in the relation between semi-linear PDEs and the related backward stochastic differential partial differential equation (BSPDE), in the rest of the paper we consider only the case when the FBSDE is decoupled. 
In this paper, we focus on decoupled FBSDEs and obtain the convergence result, Theorem \ref{Thm: Newton-Kantorovitch BSDE}. 

The rest of this paper is structured as follows. 
Section \ref{Preliminaries} introduces the notations and assumptions used. 
Section \ref{Newton-Kantorovitch  scheme for BSDEs} is devoted to 
formulating the Newton-Kantorovitch approximation process. 
Finally, 
Section \ref{Convergence of the Newton-Kantorovitch  approximation process for BSDEs} 
proves the main theorem for the decoupled FBSDEs. 

\section{Preliminaries}\label{Preliminaries}
For $x, y \in \mathbb{R}^\dimX$, $|x|$ denotes the Euclidian norm 
and $\langle x, y \rangle$ denotes the inner product. 
Matrixes size of $\dimZ$ will be represented as an element $y \in \mathbb{R}^{\dimZ}$,  
whose Euclidean norms are also given by $|y| = \sqrt{y y^\top}$ and for which $\langle y, z\rangle = \mathrm{trace}(yz^\top)$ 
where $A^\top$ is the transpose matrix of $A$. 
In this paper, we only consider the derivatives with respect to the space variable. 

We also use the following notation, based on that in \cite{Pardoux2014zbMATH06276084}. 
Let $(\Omega , \mathcal{F}, \mathbb{P})$ be a complete probability space and $\left( \mathcal{F}_t \right)_{t \geq 0 }$ be a Brownian filtration. 
$\mathcal{P} =\mathcal{P}(\left( \mathcal{F}_t \right)_{t \geq 0 })$ is the $\sigma$-algebra of 
the progressively measurable subset % 20180505, \cite[page xv]{Pardoux2014zbMATH06276084}. 
$A \subset \Omega \times [0, \infty)$ such that, for all $t \geq 0$, $A \cap ( [0, t] \times \Omega) \in \mathcal{B} [0,t] \otimes  \mathcal{F}_t $. 
%and $m \mathcal{P}$ ($cm \mathcal{P}$)stands for the collection of $\mathcal{P}$-measurable functions can be continuously modified. %\cite[pagge 18]{Pardoux2014zbMATH06276084}. 
Let $T>0$ be a fixed, and final, deterministic time. 
For $\dimY, \dimB \in \mathbb{N}$, we define 
\begin{align*}
	\mathbb{L}^{2}%=L^{p}_{\dimY}\left(\Omega , \mathcal{F}_T, \mathbb{P} \right)
	&=  \left\{ \xi : \Omega  \to \mathbb{R}^\dimY 
	\text{ is $\mathcal{F}_T$-measurable and }
	\| \xi\|_{\mathbb{L}^2}= \left\{ \mathbb{E} \left[ |\xi |^2 \right] \right\}^{\frac{1}{2}}  <\infty \right\},\\
	\mathbb{L}^{2}_{T}
	&=  \left\{ Y : \Omega \times [0,T]  \to \mathbb{R}^\dimY
	\text{ continuous}
	:Y(t) \in m\mathcal{F}_T, \forall t \in [0,T],~
	%\text{ is $\mathcal{P}(\left( \mathcal{F}_t \right)_{t \geq 0 })$-measurable and }
	%   \mathcal{B} [0,t]  \otimes \mathcal{F}_t:Y(t) \in m\mathcal{F}_T, \forall t \in [0,T],~
	\| Y \|_{\mathbb{L}_T^{2}}
		%=\left\{ \mathbb{E} \left[  \sup_{0 \leq s \leq T} |Y(s)|^2 \right] \right\}^{\frac{1}{2}}
		 <\infty \right\},\\
	\mathbb{S}^{2}_{\dimY}%=S^{p}_{\dimY}[0,T] 
	&=  \left\{ Y : \Omega \times [0,T]  \to \mathbb{R}^\dimY 
	%\in \mathbb{L}^{2}_{T} 
	\text{ continuous, adapted}:
	%\in cm\mathcal{P} :  % \text{ is progressive measurable with values in $\mathbb{R}^d$ satisfying } 
		%{\| Y \|}_{S} =
		\| Y \|_{\mathbb{S}^{2}_{\dimY}}<\infty 
		%\left\{ \mathbb{E} \left[  \sup_{0 \leq s \leq T} |Y(s)|^2 \right] \right\}^{\frac{1}{2}} <\infty
		 \right\},
\\ 	\mathbb{H}^{2}%=\mathbb{H}^{p}_{\dimZ}(0, T) 
	&=  \left\{ Z : \Omega \times [0,T]  \to \mathbb{R}^{\dimZ}
	\text{ continuous, adapted}:
	%\text{ is $\mathcal{P}(\left( \mathcal{F}_t \right)_{t \geq 0 })$-measurable and }
	%\in m\mathcal{P}: 
	%{\| Z \|}_{\mathbb{H}} =
	{\| Z \|}_{\mathbb{H}^2}
	%= \left\{  \mathbb{E} \left[ \int_0^T  |Z(s)|^2 \ds   \right] \right\}^{\frac{1}{2}}
	<\infty  \right\},
\end{align*}
where the norms $\|\cdot\|_{\mathbb{L}_T^{2}}$, $\|\cdot\|_{\mathbb{S}^{2}_{\dimY}}$, and ${\| \cdot \|}_{\mathbb{H}^2}$ are defined by
\begin{align*}
	\| Y \|_{\mathbb{L}_T^{2}}
	=\| Y \|_{\mathbb{S}^{2}_{\dimY}}
	=\left\{ \mathbb{E} \left[  \sup_{0 \leq s \leq T} |Y(s)|^2 \right] \right\}^{\frac{1}{2}},
	\quad
	{\| Z \|}_{\mathbb{H}^2}
	= \left\{  \mathbb{E} \left[ \int_0^T  |Z(s)|^2 \ds   \right] \right\}^{\frac{1}{2}}.
\end{align*}
For the sake of simplicity, we also write the operator norm of the operator $A$ as $\| A \|$.
The Banach spaces $\mathbb{S}^{2}_{\dimY} \times \mathbb{H}^{2}$ and $\mathbb{S}^{2}_{\dimX}$ are defined by 
\[
	\| (Y,Z)\|^2 = \|Y\|_{\mathbb{S}^{2}_{\dimY}}^2+\| Z\|_{\mathbb{H}^2}^2, \quad 
	\| X\|^2 = \|X\|_{\mathbb{S}^{2}_{\dimX}}^2. 
\]
For $\alpha \in \mathbb{R}$, 
we introduce the weighted norm 
\[
	\| (Y,Z)\|_{\alpha}^2 =
	%\| e^{\alpha \cdot } Y \|_{\mathbb{S}^p} +{\|e^{\alpha \cdot }  Z \|}_{\mathbb{H}^p} = 
	\mathbb{E} \left[  \sup_{0 \leq s \leq T}e^{\alpha s} |Y(s)|^2 \right]
	+ \mathbb{E} \left[ \int_0^T  e^{\alpha s}|Z(s)|^2 \ds   \right].
\]
For $p,q \in \mathbb{N}$, $C^k (\mathbb{R}^p, \mathbb{R}^q)$ is the set of functions of class $C^k$ from $\mathbb{R}^p$ to $\mathbb{R}^q$, 
and $C^k_b (\mathbb{R}^p, \mathbb{R}^q)$ is the subset of these functions whose partial derivatives of order at most of the $k$ values are bounded.  
When the domain and range dimensions are clear based on context and 
when there is no risk of confusion, we often eliminate the spaces to simplify the notation. 

For a smooth $g$ such that $g(\cdot,\cdot, \cdot) \in C^1 (\mathbb{R}^p \times \mathbb{R}^q \times \mathbb{R}^r , \mathbb{R}^q)$, it is convenient to obtain a concrete representation of ${g}' : \mathbb{R}^{p} \times \mathbb{R}^{q} \times \mathbb{R}^{r} \to \mathbb{R}^{q}$.
The Fr\'echet derivative  at $u\in \mathbb{R}^{p} \times \mathbb{R}^{q} \times \mathbb{R}^{q}$ is the matrix representation of $g'(u)$ \cite[page 60]{Ortega1970MR0273810} 
and it can be obtained using the Jacobian matrix: 
\begin{equation}
\begin{split}\label{notation: The Mean value theorem in the Banach space}
	g(u+\Delta u )-g(u)={g}'  (u)\Delta u +\mathrm{R}_{g}\left(u \right)\Delta u , \quad \Delta u  \in \mathbb{R}^{p} \times \mathbb{R}^{q} \times \mathbb{R}^{r}
\end{split}
\end{equation}
where the Lagrange remainder is given by 
\begin{align*}
\mathrm{R}_{g} (u)\Delta u  =
	\left( \int_0^1  \left\{ {g}'  \left(u+ \theta \Delta u \right) - {g}'  (u) \right\}\mathrm{d}\theta \right)\Delta u.
\end{align*}
Notice that, for all $u=(x,y,z)^\top , \ \Delta u = (\Delta x,\Delta y ,\Delta z )^\top\in\mathbb{R}^p \times \mathbb{R}^q \times \mathbb{R}^{r}$, 
\begin{equation}\label{notation: linearity of the Jacobian matrix}
	{g}' (u)\Delta u= {g}'_x (u) \Delta x +{g}'_y (u) \Delta y  +{g}'_z (u) \Delta z, 
\end{equation}
where we obtain set ${g}'_x (u):\mathbb{R}^p \to \mathbb{R}^p \times \mathbb{R}^p$ such that ${g}'_x (u) \Delta x =\left( \frac{\partial}{\partial x_j} g^{(i)} (x,y,z) \Delta x^{(i)}   \right)_{i \leq p, \ j \leq p}$ where $\Delta x^{(i)}$ stands for the $i$th component for $i \leq p$. 
We define ${g}'_y$ and ${g}'_z$, similarly. 
Finally, for all $s \in [0,T]$, we define 
\begin{equation*}%\label{def: Jacobian operator}
	 \| {g}'  \|_{\infty}
	  = \sup_{(s,x,y,z) \in [0,T] \times \mathbb{R}^{p} \times \mathbb{R}^{q} \times \mathbb{R}^{r} } | {g}'(s,x,y,z) |, 
\end{equation*}
for $g(s, \cdot, \cdot, \cdot) \in C^1_b (\mathbb{R}^p \times \mathbb{R}^q \times \mathbb{R}^r, \mathbb{R}^q)$.

\section{Newton-Kantorovitch scheme for FBSDEs}\label{Newton-Kantorovitch  scheme for BSDEs}
In this section, we formulate the Newton-Kantorovitch scheme of the following system of FBSDEs: 
\begin{equation}\label{bsde0}
\begin{split}
X (t) &= X(0)+ \int_0^t  b (s,X(s))  \ds +\int_0^t \sigma (s, X (s))  \dBs, \\
Y (t) &=  \varphi \left(X(T) \right) +\int_t^T f (s,X(s),Y(s),Z(s))  \ds - \int_t^T Z(s)\dBs, \\
\end{split}
\end{equation}
where the solutions $(X,Y,Z) = \left(X(t),Y(t), Z(t) \right)_{t \in [0, T]}$ take values in $\mathbb{R}^{\dimX} \times \mathbb{R}^{\dimY} \times \mathbb{R}^{\dimZ}$, $W$ is the $\dimB$-dimensional Wiener process, and the (progressively) measurable functions $b$, $f$, $\sigma$ and $\varphi$ are defined on the probability space $(\Omega, \mathcal{F}, \left(\mathcal{F}_t\right)_{t \geq 0}, \mathbb{P})$.%, and $\xi: \Omega \to \mathbb{R}^\dimY$ is the final condition. 

In this paper, we also assume the following. 
\begin{itemize}
\item $X(0): \Omega \to \mathbb{R}^\dimX$ is an $\mathcal{F}_0$-measurable and square-integrable random vector. 
\item $b : [0,T] \times \Omega \times \mathbb{R}^{\dimX}  \to  \mathbb{R}^{\dimX}$ 
	is $\mathcal{P} \otimes \mathcal{B} (\mathbb{R}^{\dimX})$-measurable. % and $\mathcal{B} (\mathbb{R}^{\dimY})$.
\item $\sigma: [0,T]  \times \Omega  \times \mathbb{R}^{\dimX}  \to  \mathbb{R}^{\dimX \times \dimB} $
	is 
	$\mathcal{P} \otimes \mathcal{B} (\mathbb{R}^{\dimX})$-measurable. 
\item There exists a constant $C>0$ such that, for any $x \in \mathbb{R}^{\dimX}$,
\begin{align*}
	| b (s,x) | &\leq |b (s,0) |+ C |x|, \quad \text{$(s,\omega)$-a.e., }  \\
	| \sigma (s,x) | &\leq |\sigma (s,0)|+C|x|, \quad \text{$(s,\omega)$-a.e. }  
\end{align*}
\item $\varphi : \Omega \times \mathbb{R}^{\dimX} \to \mathbb{R}^{\dimY}$ is $\mathcal{F}_T \otimes \mathcal{B} (\mathbb{R}^{\dimX})$-measurable.
\item $f: [0,T] \times \Omega \times \mathbb{R}^{\dimX} \times  \mathbb{R}^{\dimY} \times \mathbb{R}^{\dimZ} \to  \mathbb{R}^{\dimY}$ 
	is $\mathcal{P} \otimes \mathcal{B} (\mathbb{R}^{\dimY}) \otimes \mathcal{B} ( \mathbb{R}^{\dimZ} )$-measurable. % and $\mathcal{B} (\mathbb{R}^{\dimY})$.
\item 
There exists a constant $C>0$ such that for any $(x,y,z) \in \mathbb{R}^{\dimX} \times \mathbb{R}^{\dimY} \times \mathbb{R}^{\dimZ}$,
\begin{align*}
	| f (s,x,y,z) | &\leq |f (s,0,0,0) |+ C \left( |x|+|y|+ |z| \right),\quad \text{$(s,\omega)$-a.e. }\\
	| \varphi(x)| &\leq |\varphi(0)| + C|x|,\quad \text{$\omega$-a.e. }
\end{align*}
\end{itemize}
In addition, in the BSDE,
\begin{equation*}
\begin{split}
Y (t) &=  \xi +\int_t^T f (s,Y(s),Z(s))  \ds - \int_t^T Z(s)\dBs, 
\end{split}
\end{equation*}
we replace the above assumption on $\varphi$ with the following condition: 
\begin{itemize}
\item $\xi: \Omega \to \mathbb{R}^\dimY$ is an $\mathcal{F}_T$-measurable and square integrable random vector; 
$\xi \in \mathbb{L}^2$. 
\end{itemize}

We also introduce the following assumptions. 
\begin{Ass}\label{ass0}
	$b(\cdot,0),
	\sigma(\cdot,0),
	f(\cdot ,0,0,0) \in \mathbb{H}^2$, i.e., 
\begin{align*}
	\mathbb{E} \left[ \int_0^T \left\{ | b (s,0) |^2 + | \sigma (s,0) |+| f (s,0,0,0) |^2 \right\}\ds \right] <\infty
\end{align*}
and $\varphi(0) \in \mathbb{L}^2$.
\end{Ass}

\begin{Ass}\label{ass0+1}
$b (s, \cdot) \in C^1_b (\mathbb{R}^\dimX , \mathbb{R}^\dimX  )$,
$\sigma (s,\cdot )  \in C^1_b (\mathbb{R}^\dimX , \mathbb{R}^{\dimX \times \dimB}  )$, $f(s, \cdot, \cdot, \cdot) \in C^1_b (\mathbb{R}^\dimX \times \mathbb{R}^\dimY \times \mathbb{R}^\dimZ, \mathbb{R}^\dimY  ) , \ 
%g(s, \cdot, \cdot) \in C^1_b (\mathbb{R}^\dimY \times \mathbb{R}^\dimZ, \mathbb{R}^\dimZ  ), \ 
\text{$(s,\omega)$-a.e.}$
and $\varphi \in C^{1}_{b}(\mathbb{R}^{\dimX}, \mathbb{R}^{\dimY})$, $\omega$-a.e.
\end{Ass}

For given $\varphi$, we consider the operator $\NewtonF_{\varphi}$ (defined by \eqref{Def: F in BSDEs}), namely, 
\begin{equation*}
\begin{split}
	\NewtonF_{\varphi} (u)(t)
	= 
	\begin{pmatrix}
	\displaystyle x(t) - x(0) -\int_0^t b(s, x(s)) \ds - \int_0^t  \sigma(s,x(s)) \dBs\\
	\displaystyle y (t) - \varphi(x(T))  -\int_t^T f (s, u(s)) \ds + \int_t^T  z (s) \dBs 
	\end{pmatrix}^\top, 
\end{split}
\end{equation*}
for $u=(x,y, z)\in \mathbb{S}^{2}_{\dimX} \times \mathbb{S}^{2}_{\dimY} \times \mathbb{H}^2$.
As an immediate consequence of the result given Lemma 3.1 of \cite{KYnewton91zbMATH00019571}, 
we obtain the following corresponding result for FBSDEs. 

\begin{Lem}\label{op_F}
If Assumption \ref{ass0} holds, then the operator $\NewtonF_{\varphi}$ defined by \eqref{Def: F in BSDEs} maps the space $ \mathbb{S}^{2}_{\dimX} \times \mathbb{S}^{2}_{\dimY} \times \mathbb{H}^2$ into $ \mathbb{S}^{2}_{\dimX} \times \mathbb{L}_T^{2}$ and $t \mapsto \NewtonF_{\varphi}(u) (t)$ is a continuous modification for $u \in \mathbb{S}^{2}_{\dimX} \times \mathbb{S}^{2}_{\dimY} \times \mathbb{H}^2$.
\end{Lem}
\begin{proof}
For any $u=(x,y,z) \in \mathbb{S}^{2}_{\dimX} \times \mathbb{S}^{2}_{\dimY} \times \mathbb{H}^2$, it follows, from $\|x\|_{\mathbb{S}_{\dimX}^2}<\infty$ and $\| z \|^2_{\mathbb{H}^2}<\infty$, that the It\^o integrals $t \mapsto \int_0^t \sigma(s,x(s)) \dBs$ and $t \mapsto \int_t^T z(s)  \dBs$, respectively, are continuous modifications. 
By the Jensen inequality, we obtain 
\begin{align*}
	&\mathbb{E} \left[  \int_0^t \left| b (s, x(s))\right|^2  \ds \right]
	\leq 2^{}\mathbb{E} \left[  \int_0^t | b (s, 0) |^2 \ds \right]+2^{}C^2T  \|x\|^2_{\mathbb{S}^{2}_{\dimX}}<\infty
\end{align*}
and
\begin{align*}
	&\mathbb{E} \left[  \int_t^T \left| f (s, x(s), y(s), z(s))\right|^2  \ds \right]
	\\&\leq 4^{}\mathbb{E} \left[  \int_t^T | f (s, 0, 0, 0) |^2 \ds \right]+4^{}C^2  \left( T  \{\|x\|^2_{\mathbb{S}^{2}_{\dimX}}+\|y\|^2_{\mathbb{S}^{2}_{\dimY}}\}+ \| z\|^2_{\mathbb{H}^2} \right)
	<\infty. 
\end{align*}
By Doob's inequality and It\^o's isometry property,  there exists a $c>0$ such that
\begin{align*}
	&\mathbb{E} \left[ \sup_{0 \leq t \leq T} \left| \int_0^t \sigma(s,x(s))  \dBs \right|^2  \right] \\
	&\leq c \mathbb{E} \left[ \left| \int_0^T \sigma(s,x(s))  \dBs \right|^2  \right]\leq 2c \mathbb{E} \left[ \int_0^T | \sigma(s,0)|^2 \ds \right]+2c C^2 T\|x\|^2_{\mathbb{S}_{\dimX}}< \infty
\end{align*}
and
\begin{align*}
	&\mathbb{E} \left[ \sup_{0 \leq t \leq T} \left| \int_t^T z(s)  \dBs \right|^2  \right]
	\\&\leq c \mathbb{E} \left[ \left| \int_0^T z(s)  \dBs \right|^2  \right]
	= c  \mathbb{E} \left[ \int_0^T | z(s) |^2 \ds \right]< \infty. 
\end{align*}
Using the Jensen inequality, we further obtain that 
\begin{align*}
&\frac{1}{9} %20180411  9^{-1}
 \mathbb{E} \left[ \sup_{0\leq t \leq T}| \NewtonF_{\varphi}(u)(t)|^2 \right]
\\&\leq
\|x\|^2_{\mathbb{S}^{2}_{\dimX}}
+\mathbb{E}[|x(0)|^2]
+\mathbb{E} \left[  \int_0^t \left| b (s, x(s))\right|^2  \ds \right]
+\mathbb{E} \left[ \sup_{0 \leq t \leq T} \left| \int_0^t \sigma(s,x(s))  \dBs \right|^2  \right] \\
&+\|y\|^2_{\mathbb{S}^{2}_{\dimY}}
+\mathbb{E}\left|[\varphi(0)|^2 + C|x(T)|^2 \right]
\\&+\mathbb{E}\left[ \left| \int_t^T f (s, x(s), y(s), z(s)) \ds \right|^2 \right]
+ \mathbb{E}\left[\sup_{0 \leq t \leq T} \left| \int_t^T z(s)  \dBs \right|^2  \right],
\end{align*}
which is bounded by the above estimates. 
This completes the proof.
\end{proof}

The following lemma shows that the operator $\NewtonF_{\varphi}: \mathbb{S}^{2}_{\dimX} \times \mathbb{S}^{2}_{\dimY} \times \mathbb{H}^2 \to \mathbb{S}_{\dimX}^2 \times \mathbb{L}_T^{2}$ has a G\^ateaux derivative.
\begin{Lem}\label{Lem: the Gateaux derivative on BSDEs} 
If Assumptions \ref{ass0} and \ref{ass0+1} hold, then, for all $u=(x,y,z) \in \mathbb{S}^{2}_{\dimX} \times \mathbb{S}^{2}_{\dimY} \times \mathbb{H}^2$, 
the G\^ateaux derivative $\NewtonF_{\varphi}'(u): \mathbb{S}^{2}_{\dimX} \times \mathbb{S}^{2}_{\dimY} \times \mathbb{H}^2 \to \mathbb{S}_{\dimX}^2 \times \mathbb{L}_T^2$ of $\NewtonF_{\varphi}$ at $u \in \mathbb{S}^{2}_{\dimX} \times \mathbb{S}^{2}_{\dimY} \times \mathbb{H}^2$ in the direction $\overline{u}=(\overline{x},\overline{y},\overline{z})\in \mathbb{S}^{2}_{\dimY} \times \mathbb{H}^2$ exists and is given for any $t \in [0,T]$, by 
\begin{align}\label{G_derivative}
	\NewtonF_{\varphi}'(u) \overline{u}(t)
	=
	\begin{pmatrix}
	\displaystyle 
	\overline{x} (t)-\overline{x}(0)
	-\int_0^t b'_{x}  (s,x(s)) \overline{x} (s) \ds - \int_0^t \sigma'_{x} (s,x(s)) \overline{x} (s)  \dBs \\
	\displaystyle \overline{y}(t)
	-\varphi'(x(T)) \overline{x}(T)
	-\int_t^T {f}'  (s,u(s)) h (s) \ds + \int_t^T  \overline{z}(s)  \dBs
	\end{pmatrix}^{\top}.
\end{align}
%This says that the derivatives are independent of selecting the terminal condition, thus we write 
%$\NewtonF'_{\varphi} \equiv \NewtonF'_{}$. 
\end{Lem}
\begin{proof}
We denote the right-hand side of \eqref{G_derivative} by $A(u) \overline{u}(t)$ for $\overline{u}=(\overline{x}, \overline{y},\overline{z})\in \mathbb{S}^{2}_{\dimX} \times \mathbb{S}^{2}_{\dimY} \times \mathbb{H}^2$.
We note that because 
\begin{align*}
	&\mathbb{E}\left[  \int_t^T 
	\left\{
		\left|  {b}'_{x}   (s,x(s)) \overline{x}(s)  \right|^2
		+\left|  {\sigma}'_{x}   (s,x(s)) \overline{x}(s)  \right|^2
		+\left|  {f}'   (s,u(s)) h(s)  \right|^2
	\right\}
	\ds \right]\\
	&\leq (1\vee T)\{\|{b}'  \|_{\infty}^2+\|{\sigma}'  \|_{\infty}^2+\|{f}'  \|_{\infty}^2\} \| h\|^2,
\end{align*}
we obtain that $A(u) \overline{u} \in \mathbb{L}_T^{2}$ by the same argument of Lemma \ref{op_F}. 
It follows from \eqref{notation: The Mean value theorem in the Banach space} that, for all $s \in [0,T]$, 
\begin{align*}
	b(s,x(s)+ \delta \overline{x}(s))-b(s, x(s))
	&=\delta {b}'_{x} \left(s, x(s) \right) \overline{x}(s)
	+\mathrm{R}_{b} \left(x(s) \right) (\delta \overline{x})(s), \\
	\sigma (s,x(s)+ \delta \overline{x}(s))-\sigma(s, x(s))
	&=\delta {\sigma}_x'\left(s, x(s) \right) \overline{x}(s)
	+\mathrm{R}_{\sigma} \left(x(s) \right) (\delta \overline{x})(s), \\
	f(s,u(s)+ \delta \overline{u}(s))-f(s, u(s))
	&=\delta {f}'\left(s, u(s) \right) \overline{u}(s)
	+\mathrm{R}_{f} \left(u(s) \right) (\delta \overline{u})(s), 
\end{align*}
and
\begin{align*}
	\varphi(x(T)+ \delta \overline{x} (T))-\varphi(x(T))
	= \delta {\varphi}'\left(x(T) \right) \overline{x}(T)
	+\mathrm{R}_{\varphi} \left(x(T) \right)(\delta \overline{x})(T). 
\end{align*}
Hence, for all $\delta>0$ and $t \in [0,T]$,
\begin{align*}
	&
	\frac{\NewtonF_{\varphi} \left( u+  \delta \overline{u} \right) (t) - \NewtonF_{\varphi}(u)(t)}{\delta}
	\\&
	=A(u)\overline{u}(t)
	+
	\frac{1}{\delta}
	\begin{pmatrix}
		\displaystyle
		-\int_{0}^{t} \mathrm{R}_{b} \left(x(s) \right) (\delta \overline{x})(s) \ds
		-\int_{0}^{t} \mathrm{R}_{\sigma} \left(x(s) \right) (\delta \overline{x})(s) \dBs\\
		\displaystyle
		-\mathrm{R}_{\varphi} \left(x(T) \right)(\delta \overline{x})(T)
		-\int_{t}^{T}\mathrm{R}_{f} \left(u(s) \right) (\delta \overline{u})(s) \ds.
	\end{pmatrix}^{\top}. 
\end{align*}
Since $b'_x(s,\cdot)$, $\sigma'_x(s,\cdot)$, $f'(s,\cdot)$ and $\varphi'$ are bounded and continuous $(s,\omega)$-a.e., by using the dominated convergence theorem, we obtain
\begin{align*}
	\lim_{\delta \to 0}
	\left\|\frac{\NewtonF_{\varphi} \left( u+  \delta \overline{u} \right) - \NewtonF_{\varphi}(u)}{\delta}-A(u)\overline{u} \right\|_{ \mathbb{S}_{\dimX}^2 \times \mathbb{L}_T^2}=0,
\end{align*}
thus completing the proof. 
\end{proof}

The following lemma is a key in order to define the Newton-Kantorovitch approximation process.
%The following lemma is an analog of Lemma 4.1 in \cite{KYnewton91zbMATH00019571} and shows that the G\^aeaux derivative $\NewtonF'_{\varphi}(u)$ at $u \in \mathbb{S}^{2}_{\dimX} \times \mathbb{S}^{2}_{\dimY} \times \mathbb{H}^2$ has an inverse mapping.

\begin{Lem}\label{lem: derivative F's inverse estimation}
	Suppose that Assumptions \ref{ass0} and \ref{ass0+1} hold 
	and let $u=(x,y,z) \in \mathbb{S}^{2}_{\dimX} \times \mathbb{S}^{2}_{\dimY} \times \mathbb{H}^2$.
	Then, there exists a unique $\widetilde{u}=(\widetilde{x}, \widetilde{y},\widetilde{z}) \in \mathbb{S}^{2}_{\dimX} \times \mathbb{S}^{2}_{\dimY} \times \mathbb{H}^2$ such that $\NewtonF_{\varphi}(u)+\NewtonF_{\varphi}'(u) (\widetilde{u}-u)=0$ with initial condtion $\widetilde{x}(0)=x(0)$.
\end{Lem}
\begin{proof}
	Let $u=(x,y,z)$ be in $\mathbb{S}^{2}_{\dimX} \times \mathbb{S}^{2}_{\dimY} \times \mathbb{H}^2$.
	We show that there exists a unique $\widetilde{u}=(\widetilde{x}, \widetilde{y},\widetilde{z}) \in \mathbb{S}^{2}_{\dimX} \times \mathbb{S}^{2}_{\dimY} \times \mathbb{H}^2$ such that $\NewtonF_{\varphi}(u)+\NewtonF_{\varphi}'(u)(\widetilde{u}-u)=0$ with $\widetilde{x}(0)=x(0)$, i.e., 
	\begin{equation}\label{eq:linear BSDEs}
	\begin{split}
	\widetilde{x}(t)
	&=x(0)+\int_0^t  \left\{ b(s,x(s))+b'_{x}  (s,x(s))(\widetilde{x}(s)-x(s)) \right\}  \ds\\
	&\quad+\int_0^t \left\{ \sigma(s,x(s))+\sigma'_{x}  (s,x(s))(\widetilde{x}(s)-x(s)) \right\} \dBs,\\
	\widetilde{y}(t)
	&=\varphi  (x(T)) +\varphi'(x(T)) (\widetilde{x}(T)-x(T))\\
	&\quad+ \int_t^T \left\{ f(s,u(s))+{f}'  (s,u(s))(\widetilde{u}(s)-u(s)) \right\}  \ds
	-\int_t^T \widetilde{z}(s) \dBs,
	\end{split}
	\end{equation}
	Because the above equation is a linear decoupled FBSDE with uniformly bounded coefficients, 
	there exists a unique $\widetilde{u} \in \mathbb{S}^{2}_{\dimX} \times \mathbb{S}^{2}_{\dimY} \times \mathbb{H}^2$ as required. 
\end{proof}

From Lemma \ref{lem: derivative F's inverse estimation}, we can conclude that, 
for any initial condition $(X_{0}, Y_{0}, Z_{0}) \in \mathbb{S}^{2}_{\dimX} \times \mathbb{S}^{2}_{\dimY} \times \mathbb{H}^2$, we can define the Newton-Kantorovitch approximation process $(X_{n+1}, Y_{n+1}, Z_{n+1}) \in \mathbb{S}^{2}_{\dimX} \times \mathbb{S}^{2}_{\dimY} \times \mathbb{H}^2$ as solving the equation
\begin{equation}
\begin{split}\label{equation: backward Newton-Kantorovitch def_n}
\NewtonF_{\varphi}(X_{n},Y_{n},Z_{n}) +\NewtonF_{\varphi}' (X_{n},Y_{n},Z_{n}) (X_{n+1}-X_{n},Y_{n+1}-Y_{n},Z_{n+1}-Z_{n}) =0, 
\\ X_{n+1}(0)=X (0), 
\end{split}
\end{equation}
for $n \in \mathbb{N} \cup \{0\}$, which is equivalent to
\begin{align*}%\label{equation: backward Newton-Kantorovitch def_2}
(X_{n+1}, Y_{n+1},Z_{n+1})=
(X_{n}, Y_{n},Z_{n}) - {\NewtonF'_{\varphi} (X_{n}, Y_{n},Z_{n})}^{-1} \NewtonF_{\varphi}(X_{n}, Y_{n},Z_{n}).
\end{align*}

The following theorem shows that \eqref{equation: backward Newton-Kantorovitch def_n} has a unique solution that satisfies 
a linear decoupled FBSDE with uniformly bounded coefficients. 

\begin{Thm}\label{Theorem: linear BSES characerlization}
If Assumptions \ref{ass0} and \ref{ass0+1} hold.
%then, there exists a unique process $(X_{n}(t), Y_{n} (t), Z_{n}(t))_{t \in [0,T]} \in \mathbb{S}^2_{\dimX} \times \mathbb{S}^2_{\dimY} \times \mathbb{H}^2$ given by \eqref{equation: backward Newton-Kantorovitch def_n}. 
%In addition, it is compatible with the following linear decoupled FBSDE for $0\leq t \leq T$: 
Then $(X_{n+1}, Y_{n+1},Z_{n+1})$ satisfies the following linear decoupled FBSDE for $0\leq t \leq T$:
\begin{equation}\label{bsde1+Newton}
\begin{split}
	X_{n+1}(t)&=X_{n+1}(0)+\int_0^t b_n(s,X_{n+1}(s))\ds +\int_0^t \sigma_n(s,X_{n+1} (s)) \dBs,\\
	Y_{n+1}(t)&=\varphi_n(X_{n+1}(T))+ \int_t^T f_n (s, X_{n+1}(s),  Y_{n+1}(s), Z_{n+1}(s))  \ds - \int_t^T Z_{n+1}(s) \dBs, 
\end{split}
\end{equation}
where we define for $0 \leq s \leq T$ and $(x,y,z) \in \mathbb{R}^\dimX \times \mathbb{R}^\dimY \times \mathbb{R}^{\dimZ}$,
\begin{align*}
	b_n(s,x)&=b(s,X_n (s))+b'_{x}  (s,X_n (s)) (x - X_{n}(s)),\\
	\sigma_n(s,x)&=\sigma(s,X_n (s))+\sigma'_{x}  (s,X_n (s)) (x - X_{n}(s)),\\
	f_n (s, x, y, z) &=f (s, X_{n}(s), Y_{n}(s),Z_{n}(s)),\\
	&\quad+ {f}'   (X_{n}(s), Y_{n}(s), Z_{n}(s)) (x-X_{n}(s), y-Y_{n}(s) ,z-Z_{n}(s)),\\
	\varphi_n(x)&=\varphi(X_n(T))+\varphi'_{x}  (X_n (T)) (x - X_{n}(T)).
\end{align*}
In particular, if $X_0(0)=X(0)$, then $X_{n+1}(0)=X(0)$ for all $n \in \mathbb{N} \cup \{0\}$.
\end{Thm}

\begin{proof}
The existence and uniqueness of $(X_{n+1} (t), Y_{n+1} (t), Z_{n+1}(t))_{t \in [0,T]}$ follows from Lemma \ref{lem: derivative F's inverse estimation}, where, for all $n \in \mathbb{N} \cup \{0\}$.
The equation \eqref{eq:linear BSDEs} can be re-expressed as the desired FBSDE for $0\leq t \leq T$.
Finally, by the initial condition $X_{n+1}(0)=X_{n}(0)$, if $X_0(0)=X(0)$, then $X_{n+1}(0)=X(0)$, for all $n \in \mathbb{N} \cup \{0\}$.
%\begin{align}\label{equation: backward Newton-Kantorovitch def_2}
%(X_{n+1}, Y_{n+1},Z_{n+1})=
%(X_{n}, Y_{n},Z_{n}) - {\NewtonF'_{\varphi} (X_{n}, Y_{n},Z_{n})}^{-1} \NewtonF_{\varphi}(X_{n}, Y_{n},Z_{n}). 
%\end{align}
%Using Lemma \ref{Lem: the Gateaux derivative on BSDEs} and \ref{lem: derivative F's inverse estimation}, 
%this can be re-expressed as the desired BSDE for $0\leq t \leq T$. 
%Finally, \eqref{eq:linear BSDEs} and \eqref{equation: backward Newton-Kantorovitch def_2} imply that if $X_0(0)=X(0)$, then $X_{n+1}(0)=X_{n}(0)$, for all $n \in \mathbb{N} \cup \{0\}$.
\end{proof}

\section{Convergence of the Newton-Kantorovitch approximation process for FBSDEs}\label{Convergence of the Newton-Kantorovitch approximation process for BSDEs}
In this section, we investigate the convergence of the Newton-Kantorovitch approximation processes defined by \eqref{equation: backward Newton-Kantorovitch def_n}. 

\subsection{Convergence of the Newton-Kantorovitch approximation process for SDEs}\label{Convergence of the Newton-Kantorovitch approximation process for SDEs}
In this subsection, we consider the Newton-Kantorovitch scheme for the forward process $X$ 
and extend earlier studies of Newton-Kantorovitch methods for SDEs, 
\begin{equation}\label{SDE1}
	X(t)=X(0) + \int_0^t b(s,X(s))\ds +\int_0^t \sigma(s,X(s))\dBs. 
\end{equation}

\begin{Thm}\label{convergence_SDE}
	Let $X$ be a solution of \eqref{SDE1}.
	If Assumptions \ref{ass0} and \ref{ass0+1} hold, then, for any $X_{0} \in \mathbb{S}^{2}_{\dimX}$ where $X_0(0)=X(0)$, we obtain that, 
	for all $ n \in \mathbb{N}\cup\{0\}$ and $\epsilon \in (0,1)$, 
	\begin{equation}\label{convergence_SDE_1}
	\begin{split}
	\| X - X_{n+1}  \|^2
	&\leq \frac{C_0^{n+1}}{(n+1)!} \| X - X_{0}  \|^2
	\leq \epsilon^{n+1}  e^{C_0T/\epsilon}\| X - X_{0}  \|^2, 
	\end{split}
	\end{equation}
	where the constant $C_0$ is given by 
	\begin{align*}
	&C_0 =8  c_{b,\sigma} T \exp(4 c_{b,\sigma} T), 
	\quad c_{b,\sigma}=\| b' \|_{\infty}+18\|\sigma' \|_{\infty }+\| \sigma' \|^2_{\infty }. 
	\end{align*}
\end{Thm}

\begin{Rem}
An estimate was initially given by \cite{KYnewton91zbMATH00019571} for one-dimensional SDEs with uniformly bounded coefficients 
and an alternative estimation was proposed by \cite{Amano09zbMATH05574278}.
\end{Rem}

\begin{proof}[Proof of Theorem \ref{convergence_SDE}]
		The proof is based on \cite{Amano09zbMATH05574278}.
		By a fundamental result in \cite{Ito1942}, $X$ exists and is unique. 
		For $0\leq s \leq T$ and $n \in \mathbb{N}$, define 
		\[
		\overline{X}_{n}(s) = X(s)-X_{n}(s), 
		\]
		we obtain for all $s \in [0,T]$ and $n \in \mathbb{N}$, $\mathbb{P}$-a.s.  
		By the mean value theorem, for $s \in [0,T]$ and $n \in \mathbb{N}$, we obtain 
		\begin{align*}
		&b (s, X(s))  - b_n(s, X_{n+1}(s))  
		\\&=\left\{ b (s, X(s)) -b (s, X_{n}(s)) -b'_{x}  (s,X_{n}(s)) \overline{X}_{n}(s)\right\} 
		+b'_{x}  (s,X_{n}(s)) \overline{X}_{n+1}(s)
		\\&=
		\mathrm{R}_{b}\left(X_{n}(s) \right) \overline{X}_{n}(s)
		+b'_{x}  (s,X_{n}(s)) \overline{X}_{n+1}(s)
		\end{align*}
		and 
		\begin{align*}
		&\sigma (s, X(s))  - \sigma_n(s, X_{n+1}(s))  
		\\&=\left\{ \sigma (s, X(s)) -\sigma (s, X_{n}(s)) -\sigma'_{x}  (s,X_{n}(s))\overline{X}_{n}(s)\right\} 
		+\sigma'_{x}  (s,X_{n}(s))\overline{X}_{n+1}(s)
		\\&=
		\mathrm{R}_{\sigma}\left(X_{n}(s) \right) \overline{X}_{n}(s)+\sigma'_{x}  (s,X_{n}(s)) \overline{X}_{n+1}(s).
		\end{align*}
		Recall \eqref{notation: The Mean value theorem in the Banach space} and 
		we have, for $s \in [0,T]$, $n \in \mathbb{N}$ and $h \in \mathbb{S}^{2}_{\dimX}$, 
		\begin{align*}
		\mathrm{R}_{b}\left(X_{n}(s) \right) h(s)&=
		\left\{ \int_0^1   b'_{x} \left(s,( X_{n}(s) )+ \theta h(s) \right) \mathrm{d}\theta -  b'_{x} (s, X_{n}(s) ) \right\} h(s), 
		\\ \mathrm{R}_{\sigma}\left(X_{n}(s) \right) h(s)&=
		\left\{ \int_0^1  \sigma'_{x}  \left(s,( X_{n}(s) )+ \theta h(s) \right) \mathrm{d}\theta -  \sigma'_{x} (s, X_{n}(s) ) \right\} h(s). 
		\end{align*}
		This allows us to obtain 
		\begin{equation*}
		\begin{split}
		\overline{X}_{n+1}(t)
		&= \int_0^t  \left\{ b'_{x} (s,X_{n}(s)) \overline{X}_{n+1}(s) 
		+\mathrm{R}_{b}\left(X_{n}(s) \right)\overline{X}_{n}(s)\right\} \ds 
		\\&+ \int_0^t  \left\{ \sigma'_{x}(s,X_{n}(s)) \overline{X}_{n+1}(s) 
		+\mathrm{R}_{\sigma}\left(X_{n}(s) \right)\overline{X}_{n}(s) \right\} \dBs, 
		\end{split}
		\end{equation*}
				which, by applying It\^o's formula yields 
				\begin{equation*}
				\begin{split}
				|\overline{X}_{n+1}(t)|^2
				&= 2\int_0^t  \langle \overline{X}_{n+1}(s), b'_{x} (s,X_{n}(s)) \overline{X}_{n+1}(s)
				+\mathrm{R}_{b}\left(X_{n}(s) \right)\overline{X}_{n}(s) \rangle \ds 
				\\&+ 2\int_0^t  \langle \overline{X}_{n+1}(s), \sigma'_{x}(s,X_{n}(s)) \overline{X}_{n+1}(s) 
				+\mathrm{R}_{\sigma}\left(X_{n}(s) \right)\overline{X}_{n}(s) \dBs \rangle
				\\&+\int_0^t  \left| \sigma'_{x}(s,X_{n}(s)) \overline{X}_{n+1}(s) 
				+\mathrm{R}_{\sigma}\left(X_{n}(s) \right)\overline{X}_{n}(s) \right|^2 \ds.
				\end{split}
				\end{equation*}
		Note that we obtain 
		\begin{align*}
		2\langle \overline{X}_{n+1}(s), b'_{x} (s,X_{n}(s)) \overline{X}_{n+1}(s)
				+\mathrm{R}_{b}\left(X_{n}(s) \right)\overline{X}_{n}(s) \rangle 
				\\ \leq 2 \|b' \| |\overline{X}_{n+1}(t)|^2 +4 \|b' \||\overline{X}_{n}(t)|^2 
		\end{align*}
		and 
		\begin{align*}
		 \left| \sigma'_{x}(s,X_{n}(s)) \overline{X}_{n+1}(s) 
				+\mathrm{R}_{\sigma}\left(X_{n}(s) \right)\overline{X}_{n}(s) \right|^2
				\\ \leq 2 \|\sigma' \|^{2} |\overline{X}_{n+1}(t)|^2 +4 \|\sigma' \|^{2}|\overline{X}_{n}(t)|^2. 
		\end{align*}
		The Burkholder-Davis-Gundy's inequality implies that there exists a $c_0$ such that 
		\begin{align*}
			&\mathbb{E}\left[ \left| 2\int_0^t  \langle \overline{X}_{n+1}(t), \sigma'_{x}(s,X_{n}(s)) \overline{X}_{n+1}(s) \dBs \rangle\right| \right]
			\\&\leq  \mathbb{E}\left[
			2\left( \frac{1}{4}\sup_{0\leq s \leq t}|\overline{X}_{n+1}(s)|^2 \right)^{1/2}
			 \left(  4 c_0^2 \|\sigma' \| \int_0^t |X_{n+1}(s) |^2 \ds
			 \right)^{1/2}\right]
			\\ &\leq 
			\frac{1}{4}\mathbb{E}\left[\sup_{0\leq s \leq t}|\overline{X}_{n+1}(s)|^2\right]
			 +
			 4 c_0^2 \|\sigma' \| \mathbb{E}\left[ \int_0^t |X_{n+1}(s) |^2 \ds \right],
		\end{align*}
		where we obtain the last inequality by applying the inequality $ab \leq (a^2/2)+(b^2/2)$ for all $a,b \in \mathbb{R}$. Similarly, we have%20180411
		\begin{align*}
			&\mathbb{E}\left[ \left| 2\int_0^t  \langle \overline{X}_{n+1}(t), \mathrm{R}_{\sigma}\left(X_{n}(s) \right)\overline{X}_{n}(s) \dBs \rangle\right| \right]
			\\&\leq  \mathbb{E}\left[
			2\left( \frac{1}{4}\sup_{0\leq s \leq t}|\overline{X}_{n+1}(s)|^2 \right)^{1/2}
			 \left(  8 c_0^2 \|\sigma' \| \int_0^t |X_{n}(s) |^2 \ds
			 \right)^{1/2}\right]
			\\ &\leq 
			\frac{1}{4}\mathbb{E}\left[\sup_{0\leq s \leq t}|\overline{X}_{n+1}(s)|^2\right]
			 +
			 8 c_0^2\|\sigma' \| \mathbb{E}\left[ \int_0^t |X_{n}(s) |^2 \ds \right],
		\end{align*}
 		where we note that an explicit upper bounded of $3$ can be obtained for the constant $c_0$; see Theorem 3.28 in \cite{Karatzasbook}. 
		By setting 
		$
			c_{b,\sigma}=\|b' \|_{ \infty }+2 c_0^2\|\sigma' \|_{\infty }+\|\sigma' \|^2_{\infty },
		$ 
		 we obtain, for any $t' \in [0,T]$,
		\begin{equation*}
		\begin{split}
		\mathbb{E}\left[\sup_{0\leq t \leq t'}|\overline{X}_{n+1}(t)|^2\right]
		&\leq
		4c_{b,\sigma} \int_0^{t'}
			\mathbb{E}\left[\sup_{0\leq t \leq s}|\overline{X}_{n+1}(t)|^2\right]
		\ds
		\\&+8c_{b,\sigma}\int_0^{t'}
			\mathbb{E}\left[\sup_{0\leq t \leq s}|\overline{X}_{n}(t)|^2\right]
		\ds. 
		\end{split}
		\end{equation*}
		Gronwall's inequality further implies that
		\begin{equation}\label{convergence_SDE_2}
		\begin{split}
		\mathbb{E}\left[\sup_{0\leq t \leq t'}|\overline{X}_{n+1}(t)|^2\right]
		&\leq
		C_0\int_0^{t'} \mathbb{E}\left[\sup_{0\leq t \leq s}|\overline{X}_{n}(t)|^2\right]
		\ds,
		\end{split}
		\end{equation}
		where 
		$C_{0}%2018041
		=8c_{b,\sigma} \exp(4c_{b,\sigma}T)$.
		Iterating \eqref{convergence_SDE_2} yields 
		\begin{equation*}
		\begin{split}
		\mathbb{E}\left[\sup_{0\leq t \leq T}|\overline{X}_{n+1}(t)|^2\right]
		&\leq
		C_{0}%2018041
		^2 \int_0^{T}\ds_1 \int_0^{s_1}\ds_2
		\mathbb{E}\left[\sup_{0\leq t \leq s_2}|\overline{X}_{n-1}(t)|^2\right]
		\\&\leq	
		C_{0}%2018041
		^{n+1} \mathbb{E}\left[\sup_{0\leq t \leq T}|\overline{X}_{0}(t)|^2\right]
		\int_0^{T}\ds_1 \int_0^{s_1}\ds_2 \cdots \int_0^{s_{n}}\ds_{n+1}
		\\&=
		\frac{(TC_{0}%2018041
		)^{n+1}}{(n+1)!} \mathbb{E}\left[\sup_{0\leq t \leq T}|\overline{X}_{0}(t)|^2\right]. 
		\end{split}
		\end{equation*}
		In  addition,  we obtain that 
		\[
		\epsilon^{ -(n+1)}\frac{(TC_{0})^{n+1}}{(n+1)!}  \leq \sum_{k=0}^{\infty}\frac{(TC_{0}/\epsilon)^{k}}{k!} =e^{C_0T/\epsilon},
		\]
		thus completing the proof. 
\end{proof}
\begin{Rem}
We can also sue the same argument to show that the approximation process is a Cauchy sequence in $\mathbb{S}^{2}_{\dimX} $. 
\end{Rem}

\subsection{Toward decoupled forward-backward stochastic differential equations}\label{Toward decoupled forward-backward stochastic differential equations}
The inequality in Theorem \ref{convergence_BSDE} below indicates that approximating the terminal condition is the key to estimating the error between the solution and the Newton-Kantorovitch approximation process. 
In this section, we consider the Newton-Kantorovitch approximation with respect to the terminal condition. 
First, we show that it 
converges with respect to the weighted norm $\|\cdot \|_{\alpha}$.

\begin{Thm}\label{convergence_BSDE} 
	Let $(X,Y,Z)$ be a solution of the FBSDE \eqref{fbsde1 decoupled}.
	If Assumptions \ref{ass0} and \ref{ass0+1} hold and $(X_{0}, Y_{0}, Z_{0}) \in {\mathbb{S}^{2}_{\dimX} \times } \mathbb{S}^{2}_{\dimY} \times \mathbb{H}^2$ such that $X_{0}(0)=X(0)$, then, for all $ n \in \mathbb{N}\cup\{0\}$, $\epsilon \in (0,1)$ and $T>0$, we obtain 
	\begin{equation}\label{Theorem: the linear convergence inequality}
	\begin{split}
	\| (
	Y - Y_{n+1}, Z - Z_{n+1} ) \|_{\alpha}^2
	\leq  \epsilon^{n+1} 
	\left\{ C_1 \| X - X_{0} \|_{{\mathbb{S}^2_{\dimX}}}^2 +\| (Y - Y_{0}, Z - Z_{0} ) \|_{\alpha}^2 \right\}
	\end{split}
	\end{equation}
where 
$C_1 \equiv C_1(\epsilon)=\epsilon^{-1} \{1+\left( 1+4(2+TC_0) \| \varphi' \|_{\infty}^2 \right) \left(1+4c_{0}^2 \right)\} e^{\alpha T+C_0/\epsilon}$, 
$c_0=3$ and $C_0$ is given by Theorem \ref{convergence_SDE} and $\alpha =2 \| {f}' \|_{\infty}+4 \| {f}' \|_{\infty}^2+12 \| {f}' \|_{\infty} \left(1+ 4 c_0^2 \right) (1 \vee T) \epsilon^{-1}$. 
\end{Thm}
\begin{proof}%[Proof of Theorem \ref{convergence_BSDE} ]
\begin{subequations}
By the fundamental result in \cite{pardouxpen1990adapted}, 
the solution $(X,Y,Z)$ exists and is unique. 
For $s \in [0,T]$ and $n \in \mathbb{N}$, we define
\begin{align*}
	h_n (s)
	=(\overline{X}_{n}(s),\overline{Y}_{n}(s),\overline{Z}_n(s))
	=(X(s)-X_{n}(s), Y(s)-Y_{n}(s), Z (s)-Z_{n}(s)), 
\end{align*}
and apply the mean value theorem, yielding 
\begin{align*}
&f (s, X(s), Y(s), Z(s))  - f_n (s, X_{n+1}(s), Y_{n+1}(s), Z_{n+1}(s))  
	\\&={f}'  (s, X_{n}(s), Y_{n}(s),  Z_{n}(s) )h_{n+1} (s)
	\\&+\left\{ f (s, X(s), Y(s), Z(s)) - f (s, X_{n}(s), Y_{n}(s), Z_{n}(s)) -{f}'  (s, X_{n}(s), Y_{n}(s),  Z_{n}(s) )h_n (s)\right\} 
	\\&={f}'  (s, X_{n}(s), Y_{n}(s),  Z_{n}(s) ) h_{n+1}(s)+\mathrm{R}_{f}\left(X_{n}(s), Y_{n}(s),  Z_{n}(s) \right) h_{n}(s). 
	%\left\{ \int_0^1  J_{f} \left(( Y_{n}(s),  Z_{n}(s) )+ \theta h_n(s) \right) \mathrm{d}\theta - {f}'  (s, Y_{n}(s),  Z_{n}(s) ) \right\} h_n(s)
\end{align*}
Recall \eqref{notation: The Mean value theorem in the Banach space} and 
we have, for $s \in [0,T]$, $n \in \mathbb{N}$ and $h \in \mathbb{S}^{2}_{\dimX} \times \mathbb{S}^{2}_{\dimY} \times \mathbb{H}^2$, 
\begin{align*}
	&\mathrm{R}_{f}\left(X_{n}(s), Y_{n}(s),  Z_{n}(s) \right) h(s)
	\\ &=
	\left\{ \int_0^1  {f}'\left(s, (X_{n}(s), Y_{n}(s),  Z_{n}(s) )+ \theta h(s) \right) \mathrm{d}\theta - {f}'  (s, X_{n}(s), Y_{n}(s),  Z_{n}(s) ) \right\} h(s). 
\end{align*}
This allows us to show that $(\overline{Y}_{n+1},\overline{Z}_{n+1})$ satisfies the following linear BSDE
\begin{equation*}
\begin{split}
	&\overline{Y}_{n+1}(t) - \overline{Y}_{n+1}(T)+\int_t^T  \overline{Z}_{n+1}(s)  \dBs 
	\\&=\int_t^T  
		\left\{ {f}' (s, X_{n}(s), Y_{n}(s), Z_{n}(s) ) h_{n+1} (s)  +\mathrm{R}_{f}\left(X_{n}(s), Y_{n}(s), Z_{n}(s) \right)h_n(s)  \right\} \ds.
\end{split}
\end{equation*}
Applying It\^o's formula to $e^{\alpha  t} \left| \overline{Y}_{n+1}(t) \right|^2$ for all $\alpha \in \mathbb{R}$, we obtain 
\begin{equation}\label{L two times exponential alpha from t to T}
\begin{split}
	&e^{\alpha  t} \left| \overline{Y}_{n+1}(t) \right|^2 -e^{\alpha  T} \left| \overline{Y}_{n+1}(T) \right|^2 
	+  \int_t^T  e^{\alpha  s} \left| \overline{Z}_{n+1}(s) \right|^2   \ds 
	\\&= \int_t^T e^{\alpha  s} (-\alpha)  \left| \overline{Y}_{n+1}(s) \right|^2   \ds
	-2 \int_t^T   e^{\alpha s} \langle {\overline{Y}_{n+1}(s)}, \overline{Z}_{n+1}(s)  \dBs \rangle
	\\&+2\int_t^T e^{\alpha  s} \langle  \overline{Y}_{n+1}(s), {f}' (s, X_{n}(s),  Y_{n}(s), Z_{n}(s) ) h_{n+1} (s)
	+\mathrm{R}_{f} \left(X_{n}(s), Y_{n}(s), Z_{n}(s) \right) h_n(s) \rangle  \ds.
\end{split}
\end{equation}
From the Cauchy-Schwarz inequality and the inequality $2a b \leq \delta^{-1} |a|^2 +\delta |b|^2$ for all $\delta>0$, we have
\begin{align}
	 &
	 2 \left| \langle \overline{Y}_{n+1}(s), f'_{x} (s, X_{n}(s), Y_{n}(s), Z_{n}(s) )\overline{X}_{n+1}(s) \rangle\right| 
	 \leq 
	 2 \| {f}' \|_{\infty}^2  \left| \overline{Y}_{n+1}(s) \right|^2+\frac{1}{2} \left| \overline{X}_{n+1}(s)  \right|^2,
	 \notag
	\\&
	2 \left|\langle \overline{Y}_{n+1}(s) ,  f'_{y} (s, X_{n}(s), Y_{n}(s), Z_{n}(s) )\overline{Y}_{n+1}(s) \rangle\right| 
	\leq 2 \| {f}' \|_{\infty} \left| \overline{Y}_{n+1}(s) \right|^2, \notag
	 \\&
	 2 \left| \langle \overline{Y}_{n+1}(s), f'_{z} (s, X_{n}(s), Y_{n}(s), Z_{n}(s) )\overline{Z}_{n+1}(s) \rangle\right| 
	 \leq 
	 2 \| {f}' \|_{\infty}^2  \left| \overline{Y}_{n+1}(s) \right|^2+\frac{1}{2} \left| \overline{Z}_{n+1}(s)  \right|^2,
	 \notag
	 \\ &
	  2 \left| \langle \overline{Y}_{n+1}(s),  \mathrm{R}_{f}\left( X_{n}(s),
	  Y_{n}(s),  Z_{n}(s) \right) h_{n}(s) \rangle \right|
	  \leq  \delta^{-1}  \left| \overline{Y}_{n+1}(s) \right|^2+ \delta\left| \mathrm{R}_{f}\left(X_{n}(s), Y_{n}(s),  Z_{n}(s) \right) h_{n}(s)\right|^2 \notag,
\end{align}
and
\begin{align}
	\left| \mathrm{R}_{f}\left(X_{n}(s), Y_{n}(s),  Z_{n}(s) \right) h_{n}(s) \right|
	 \leq  2 \| {f}' \|_{\infty} \left| h_n(s) \right|.
	  \label{Rf estimation}
\end{align}
The right-hand side of \eqref{L two times exponential alpha from t to T} is therefore less than or equal to
\begin{align*}
	&\int_t^T e^{\alpha  s} \cdot 
	\left\{
		(-\alpha)
		+ 2 \| {f}' \|_{\infty}
		+4 \| {f}' \|_{\infty}^2
		+\delta^{-1}
	\right\}
	\left|
		\overline{Y}_{n+1}(s)
	\right|^2 
	\ds
	\\&+\int_t^T
		e^{\alpha  s}
		\left\{
			\frac{1}{2} \left|\overline{Z}_{n+1}(s)\right|^2 
			+
			\frac{1}{2} \left|\overline{X}_{n+1}(s)\right|^2 %20180411 
			+ \delta
				\left|
					\mathrm{R}_{f}\left(X_{n}(s), Y_{n}(s), Z_{n}(s) \right) h_{n}(s)
				\right|^2
		\right\} \ds\\
	&-2 \int_t^T   e^{\alpha s} \langle {\overline{Y}_{n+1}(s)}, \overline{Z}_{n+1}(s)  \dBs \rangle.
\end{align*}
Setting $\alpha \equiv \alpha(\delta)=  2 \| {f}' \|_{\infty}+4 \| {f}' \|_{\infty}^2+\delta^{-1} $, we obtain 
\begin{equation}
\begin{split}
	&e^{\alpha  t} \left| \overline{Y}_{n+1}(t) \right|^2 -e^{\alpha  T} \left| \overline{Y}_{n+1}(T) \right|^2  
	+\frac{1}{2}  \int_t^T  e^{\alpha  s}
	 \left\{ \left| \overline{Z}_{n+1}(s) \right|^2   -  \left| \overline{X}_{n+1}(s) \right|^2 \right\}
	  \ds  
	\\&
	\leq  \delta \int_t^T e^{\alpha  s}  \left| \mathrm{R}_{f}\left(X_{n}(s), Y_{n}(s),  Z_{n}(s) \right) h_{n}(s) \right|^2 \ds
	- 2\int_t^T   e^{\alpha  s}   \langle {\overline{Y}_{n+1}(s)}, \overline{Z}_{n+1}(s)  \dBs \rangle.  \label{an inequality before applying the BDG inequality}
\end{split}
\end{equation}
As $(Y,Z),  (Y_{n}, Z_{n}) \in \mathbb{S}^{2}_{\dimY} \times \mathbb{H}^2$, 
the local martingale 
\[
	 \left\{ \int_0^t   e^{\alpha  s}   \langle {\overline{Y}_{n+1}(s)}, \overline{Z}_{n+1}(s)  \dBs \rangle \right\}_{t \in [0,T]} 
\] 
vanishes at $0$. 
Thus, in particular, by setting $t=0$,  we obtain 
\begin{equation}
\begin{split}
	&\mathbb{E} \left[ \int_0^T  e^{\alpha  s} 
		\left\{ \left| \overline{Z}_{n+1}(s) \right|^2   -  \left| \overline{X}_{n+1}(s) \right|^2 \right\} 
		\ds \right]
	 \\ &\leq 2 \delta \mathbb{E} \left[ \int_0^T e^{\alpha  s}  \left| \mathrm{R}_{f}\left(X_{n}(s), Y_{n}(s),  Z_{n}(s) \right) h_{n}(s) \right|^2 \ds \right]
	  +2 \mathbb{E} \left[ e^{\alpha  T} \left| \overline{Y}_{n+1}(T) \right|^2  \right] 
	  \label{Z bounded by Repsilon}.
\end{split}
\end{equation}
Note that 
\begin{align*}
	\int_t^T   e^{\alpha  s}   \langle {\overline{Y}_{n+1}(s)}, \overline{Z}_{n+1}(s)  \dBs \rangle
	= \int_0^T   1_{[t,T]}(s) e^{\alpha  s}   \langle {\overline{Y}_{n+1}(s)}, \overline{Z}_{n+1}(s)  \dBs \rangle. 
\end{align*}
Applying the Burkholder-Davis-Gundy inequality indicates that there is a universal $c_{0}$ such that 
\begin{align*}
&2 \mathbb{E} \left[\sup_{0 \leq t \leq T}\left|   \int_t^T   e^{\alpha  s}   \langle {\overline{Y}_{n+1}(s)}, \overline{Z}_{n+1}(s)  \dBs \rangle \right| \right]
	\\&\leq 2 c_{0} \mathbb{E} \left[\left( \int_0^T 1_{[t,T]}(s)  e^{2 \alpha  s}   \left| \overline{Y}_{n+1}(s) \right|^2 \left|  \overline{Z}_{n+1}(s) \right|^2 \ds \right)^{\frac{1}{2}} \right]
\\	&\leq \mathbb{E} \left[ \sup_{0 \leq t \leq T} e^{ \alpha  t/2}   \left| \overline{Y}_{n+1}(t) \right| \left( 4c_{0}^2 \int_0^T   e^{ \alpha  s}    \left|  \overline{Z}_{n+1}(s) \right|^2 \ds \right)^{\frac{1}{2}}\right],
\end{align*}
 	where we note that an explicit upper bounded of $3$ can be obtained on the constant; 
	refer to Theorem 3.28 in \cite{Karatzasbook}. 
	Hence, by considering the supremum of \eqref{an inequality before applying the BDG inequality} and using the inequality $ab \leq (a^2/2)+(b^2/2)$ for all $a,b \in \mathbb{R}$, we obtain
\begin{align*}
&	\mathbb{E} \left[ \sup_{0 \leq t \leq T} e^{\alpha  t} \left| \overline{Y}_{n+1}(t)\right|^2 
	-e^{\alpha  T} \left| \overline{Y}_{n+1}(T) \right|^2
	+\frac{1}{2}  \int_0^T  e^{\alpha  s} 
	\left\{ \left| \overline{Z}_{n+1}(s) \right|^2   -  \left| \overline{X}_{n+1}(s) \right|^2 \right\}
	  \ds \right]
	\\&\leq 
	\delta \int_0^T e^{\alpha  s}  \mathbb{E} \left[ \left| \mathrm{R}_{f}\left(X_{n}(s), Y_{n}(s),  Z_{n}(s) \right) h_{n}(s) \right|^2 \right] \ds 
	\\&+
	 \frac{1}{2}\mathbb{E} \left[ \sup_{0 \leq t \leq T} e^{ \alpha  t}   \left| \overline{Y}_{n+1}(t) \right|^2 \right]
	+2 c_{0}^2
	 \mathbb{E} \left[\int_0^T   e^{ \alpha  s} \left| \overline{Z}_{n+1}(s) \right|^2 \ds \right],
\end{align*}
which implies that
\begin{align*}
&	\mathbb{E} \left[ \sup_{0 \leq t \leq T} e^{\alpha  t} \left| \overline{Y}_{n+1}(t)\right|^2 
	-2e^{\alpha  T} \left| \overline{Y}_{n+1}(T) \right|^2
	+\int_0^T  e^{\alpha  s} \left\{ \left| \overline{Z}_{n+1}(s) \right|^2   -  \left| \overline{X}_{n+1}(s) \right|^2 \right\}
	   \ds \right]
	\\&\leq 
	2\delta \int_0^T e^{\alpha  s}  \mathbb{E} \left[ \left| \mathrm{R}_{f}\left(X_{n}(s), Y_{n}(s),  Z_{n}(s) \right) h_{n}(s) \right|^2 \right] \ds 
	+4 c_{0}^2
	 \mathbb{E} \left[\int_0^T   e^{ \alpha  s} \left| \overline{Z}_{n+1}(s) \right|^2 \ds \right]. 
\end{align*}
By then taking the inequality \eqref{Z bounded by Repsilon} into consideration, we observe  
\begin{equation*}
\begin{split}
&\mathbb{E} \left[ \sup_{0 \leq t \leq T} e^{\alpha  t} \left| \overline{Y}_{n+1}(t) \right|^2 +  \int_0^T  e^{\alpha  s} \left\{ \left| \overline{Z}_{n+1}(s) \right|^2   -  \left| \overline{X}_{n+1}(s) \right|^2 \right\}
  \ds  \right]
	\\&
	\leq 
	2\delta\left( 1+4 c_0^2
	 \right) \int_0^T e^{\alpha  s}   \mathbb{E} \left[ \left| \mathrm{R}_{f}\left(X_{n}(s), Y_{n}(s),  Z_{n}(s) \right) h_{n}(s)\right|^2\right] \ds\\
	&\quad+2\left(1+4 c_0^2
	\right)e^{\alpha  T} \mathbb{E} \left[ \left| \overline{Y}_{n+1}(T) \right|^2  \right] 
	+
	4 c_0^2 \mathbb{E} \left[\int_0^T   e^{ \alpha  s} \left| \overline{X}_{n+1}(s) \right|^2 \ds \right] 
	.
\end{split}
\end{equation*}
If we also consider the inequality \eqref{Rf estimation}, we find 
\begin{align*}	
	&\int_0^T e^{\alpha  s}   \mathbb{E} \left[ \left| \mathrm{R}_{f}\left(X_{n}(s), Y_{n}(s),  Z_{n}(s) \right) h_{n}(s) \right|^2\right] \ds
	\\&\leq  6\| {f}' \|_{\infty} (1 \vee T) 
	\left\{ \mathbb{E} \left[ \sup_{0 \leq t \leq T} e^{\alpha  t} \left| \overline{X}_{n}(t) \right|^2+\sup_{0 \leq t \leq T} e^{\alpha  t} \left| \overline{Y}_{n}(t) \right|^2 +  \int_0^T  e^{\alpha  s} \left| \overline{Z}_{n}(s) \right|^2  \ds  \right]\right\}.
\end{align*}
Selecting $\delta$ such that $  12 \| {f}' \|_{\infty} \left(1+ 4 c_0^2
	 \right) (1 \vee T)  \delta=\epsilon$,  
we obtain 
\[
	\alpha
	\equiv \alpha(\delta)=2 \| {f}' \|_{\infty}+4 \| {f}' \|_{\infty}^2+12 \| {f}' \|_{\infty} \left(1+4 c_0^2\right) (1 \vee T) \epsilon^{-1}.
\]
This leads to 
\begin{align*}
\begin{split}
	&\| (\overline{Y}_{n+1}, \overline{Z}_{n+1} ) \|_{\alpha}^2 
	\\& \leq  \epsilon \| ( \overline{X}_{n}, \overline{Y}_{n}, \overline{Z}_{n})\|_{\alpha}^2
	+2\left(1+4c_{0}^2 \right)e^{\alpha  T} \|  \overline{Y}_{n+1}(T)   \|_{\mathbb{L}^2}^2
	+ \left(1+4c_{0}^2 \right)\|  \overline{X}_{n+1}   \|_{\alpha}^2.
\end{split}
\end{align*}
Hence, using \eqref{convergence_SDE_2},
\begin{align*}
	\mathbb{E} \left[ \left| \overline{Y}_{n+1}(T) \right|^2  \right]
	&\leq \| \varphi' \|_{\infty}^2 \left\{ 4\mathbb{E} \left[ \left| \overline{X}_{n}(T) \right|^2  \right]+2\mathbb{E} \left[ \left| \overline{X}_{n+1}(T) \right|^2  \right]\right\}\\
	&\leq 2(2+TC_0)\| \varphi' \|_{\infty}^2 \mathbb{E} \left[ \left| \overline{X}_{n}(T) \right|^2  \right].
\end{align*}
Thus, we obtain 
\[
	2\left(1+4c_{0}^2 \right)e^{\alpha  T} \|  \overline{Y}_{n+1}(T)   \|_{\mathbb{L}^2}^2
	\leq 4(2+TC_0) \| \varphi' \|_{\infty}^2 \left(1+4c_{0}^2 \right)e^{\alpha  T} \|  \overline{X}_{n}   \|_{\mathbb{S}^2_{\dimX}}^2.
\]
Applying inequality \eqref{convergence_SDE_1} then yields  for $c_1 = e^{\alpha  T}+\left( 1+4(2+TC_0) \| \varphi' \|_{\infty}^2 \right) \left(1+4c_{0}^2 \right)e^{\alpha  T}$,
\begin{align*}
\begin{split}
	&\| (\overline{Y}_{n+1}, \overline{Z}_{n+1} ) \|_{\alpha}^2 
	 \leq  \epsilon \| (\overline{Y}_{n}, \overline{Z}_{n})\|_{\alpha}^2 
	 +c_1 \| X - X_{0}  \|^2  \frac{C_0^{n}}{n!}.
\end{split}
\end{align*}
For any positive sequence $\{ a_n \}$, $\{ b_n\}$ and $\epsilon >0$ 
such that $a_{n+1} \leq b_{n} + \epsilon a_{n}$ for all $n \in \mathbb{N} \cup \{0\}$, we obtain
\begin{align*}
	a_{n+1} %&\leq b_{n} + \epsilon a_{n} \\ 
	&\leq b_{n}+ \epsilon ( \epsilon a_{n-1} + b_{n-1} ) 
	= \epsilon^{2} \left( \epsilon^{-2}  b_{n} + \epsilon^{-1} b_{n-1}  +  a_{n-1} \right)
	\\ & \leq \cdots \leq  \epsilon^{n+1} \left\{ \sum_{k=0}^{n}  \epsilon^{k-(n+1)}  b_{n-k} +  a_{0} \right\} . 
\end{align*}
Replacing  $b_n =c_1 \| X - X_{0}  \|^2 \frac{C_0^{n}}{n!} $ for all $n \in \mathbb{N}$, 
we obtain 
\begin{align*}
	 \sum_{k=0}^{n}  \epsilon^{k-(n+1)} b_{n-k}
	%=	 \sum_{k=0}^{n}  \epsilon^{k-(n+1)} \frac{C_0^{n-k+1}}{(n-k+1)!} 
	\leq	 \epsilon^{-1} c_1 \| X - X_{0}  \|^2  \sum_{k=0}^{n} \frac{(C_0/\epsilon)^{k}}{k!}
	\leq  \epsilon^{-1} c_1 \| X - X_{0}  \|^2  e^{C_0/\epsilon}, \quad n \in \mathbb{N}\cup \{ 0 \} . 
\end{align*}
Thus, we obtain 
\begin{align*}
\begin{split}
	&\| (\overline{Y}_{n+1}, \overline{Z}_{n+1} ) \|_{\alpha}^2 
	 \leq  \epsilon^{n+1} \left( C_1  \| X - X_{0}  \|^2 +\| (\overline{Y}_{0}, \overline{Z}_{0} ) \|_{\alpha}^2  \right), 
\end{split}
\end{align*}
where we define
$
C_1
=\epsilon^{-1} c_1 e^{C_0/\epsilon}
=\epsilon^{-1} \{1+\left( 1+4(2+TC_0) \| \varphi' \|_{\infty}^2 \right) \left(1+4c_{0}^2 \right)\} e^{\alpha T+C_0/\epsilon}
$. 
This conclude the proof. 
\end{subequations}
\end{proof}

\begin{Rem}
We can also prove the so called ``semilocal theorem," which states 
that  the approximation process is a Cauchy sequence in $\mathbb{S}^{2}_{\dimY} \times \mathbb{H}^2$ by the same argument. For the definition of "semilocal," 
refer to \cite{Yamamoto2001}. 
\end{Rem}

We can now prove our main result based on the following theorem.

\begin{Thm}\label{Thm: Newton-Kantorovitch BSDE}
Let $(X,Y,Z)$ be a solution of the FBSDE \eqref{fbsde1 decoupled}.
If Assumptions \ref{ass0} and \ref{ass0+1} hold and $(X_{0}, Y_{0}, Z_{0}) \in \mathbb{S}^{2}_{\dimX} \times
\mathbb{S}^{2}_{\dimY} \times \mathbb{H}^2$ with $X_{0}(0)=X(0)$, 
then, there exists a $C_3>0$ such that, for all $n \in \mathbb{N} \cup \{ 0\}$, we obtain 
\begin{equation}\label{Theorem: the linear convergence inequality BSDE}
\begin{split}
	&\| (X - X_{n+1}, Y - Y_{n+1}, Z - Z_{n+1} )\|^2  
	\\ &\leq  \epsilon^{n+1} C_3 \| (X - X_{0}, Y - Y_{0}, Z - Z_{0} ) \|^2, 
\end{split}
\end{equation}
where the constant $C_3 \equiv C_3 (\epsilon)$ is bounded by the coefficients and independent of $n$. 
\end{Thm}
\begin{proof}
By inequality \eqref{convergence_SDE_1} in Theorem \ref{convergence_SDE} 
and by inequality \eqref{Theorem: the linear convergence inequality} in Theorem \ref{convergence_BSDE}, we obtain 
	\begin{equation*}
	\begin{split}
	\| X - X_{n+1}  \|^2 \leq \epsilon^{n+1}  e^{C_0T/\epsilon}\| X - X_{0}  \|^2, 
	\end{split}
	\end{equation*}
	and 
	\begin{equation*}%\label{Theorem: the linear convergence inequality}
	\begin{split}
	\| (Y - Y_{n+1}, Z - Z_{n+1} ) \|^2
	&\leq \| (Y - Y_{n+1}, Z - Z_{n+1} ) \|_{\alpha}^2
	\\& \leq  \epsilon^{n+1} 
	\left\{ C_1 \| X - X_{0} \|_{{\mathbb{S}^2_{\dimX}}}^2 +\| (Y - Y_{0}, Z - Z_{0} ) \|_{\alpha}^2 \right\},
	\end{split}
	\end{equation*}
respectively, where 
	\begin{align*}
	&C_0 =8  c_{b,\sigma} T \exp(4 c_{b,\sigma} T), 
		\quad c_{b,\sigma}=\| b' \|_{\infty}+18\|\sigma' \|_{\infty}+\| \sigma' \|^2_{\infty}, 
	\\ &C_1 \equiv C_1(\epsilon)=\epsilon^{-1} \{1+\left( 1+4(2+TC_0) \| \varphi' \|_{\infty}^2 \right) \left(1+4c_{0}^2 \right)\} e^{\alpha T+C_0/\epsilon}, \quad  c_0=3, 
	\\&\alpha =2 \| {f}' \|_{\infty}+4 \| {f}' \|_{\infty}^2+12 \| {f}' \|_{\infty} \left(1+4 c_0^2
 \right) (1 \vee T) \epsilon^{-1}. 
	\end{align*}
Defining 
\[	
	C_3 =  \left\{ (C_1 + e^{C_0T/\epsilon}) \vee e^{\alpha T} \right\},
\]
concludes the proof. 
\end{proof}

This finally allows us to prove our main theorem, Theorem \ref{Theorem: the convergence of the Newton-Kantorovitch approximation process3}. 
\begin{proof}[Proof of Theorem \ref{Theorem: the convergence of the Newton-Kantorovitch approximation process3}]
By setting $\epsilon =1/2$, the desired result can be obtained from Theorem \ref{Thm: Newton-Kantorovitch BSDE}. 
\end{proof}

\section*{Acknowledgements}
We would like to express our sincere appreciation to T.~Yamada for suggesting 
a formulation of the Newton-Kantorovitch method for BSDEs. 
This work was supported by 
JSPS KAKENHI Grant Numbers 16J00894 and 17H06833.

\section*{References}
\bibliographystyle{amsplain}
%\bibliography{../NewtonKantrovich}
\providecommand{\bysame}{\leavevmode\hbox to3em{\hrulefill}\thinspace}
\providecommand{\MR}{\relax\ifhmode\unskip\space\fi MR }
% \MRhref is called by the amsart/book/proc definition of \MR.
\providecommand{\MRhref}[2]{%
  \href{http://www.ams.org/mathscinet-getitem?mr=#1}{#2}
}
\providecommand{\href}[2]{#2}

\end{document}